\documentclass[12pt,a4paper]{amsart} 
\usepackage{amsmath}
\usepackage{amsthm}
\usepackage{amssymb}
\usepackage[english]{babel}
\usepackage{ot1patch} 
\usepackage{mathrsfs}

\newtheorem{thm}{Theorem}[section]
\newtheorem{lem}[thm]{Lemma}

\theoremstyle{remark}
\newtheorem{rem}[thm]{Remark}
\newtheorem*{rem*}{Remark}

\theoremstyle{definition}
\newtheorem{dfn}[thm]{Definition}
\newtheorem{ex}[thm]{Example}

\numberwithin{equation}{section}

\newcommand{\om}{\Omega}

\newcommand{\Rz}{\mathbb{R}}
\newcommand{\Nt}{\mathbb{N}}

\begin{document}

\title{UPC condition with parameter for subanalytic sets}

\author{Anna Denkowska \& Maciej P. Denkowski} 
\date{May 3rd 2014}

\begin{abstract}
In 1986 Paw\l ucki and Ple\'sniak introduced the notion of {\sl uniformly polynomially cuspidal} (UPC) sets and proved that every relatively compact and fat subanalytic subset of ${\Rz}^n$ satisfies the UPC condition. Herein we investigate the UPC property of the sections of a relatively compact open subanalytic set $E\subset{\Rz}^k\times{\Rz}^n$ and we show that two of the three parameters in the UPC condition can be chosen independently of the section, while the third one depends generally on the point defining the section.
\end{abstract}
\maketitle

\section{Introduction}

Let $E\subset {\Rz}^n$ be non-empty. In \cite{PP} Paw\l ucki and Ple\'sniak introduced the following notion:

\begin{dfn}
The set $E$ is said to be {\it uniformly polynomially cuspidal} (or UPC for short), if it satisfies the following {\it UPC condition}:\\
there exist positive constants $M,m,d>0$ with $d\in{\Nt}$, such that for each point $x\in\overline{E}$ one may choose a polynomial map $h_x\colon {\Rz}\to {\Rz}^n$ of degree $\leq d$, satisfying \begin{enumerate}
\item $h_x(0)=x,\> h_x((0,1])\subset E;$
\item $\mathrm{dist}(h_x(t),{\Rz}^n\setminus E)\geq Mt^m,\> t\in[0,1],\> x\in\overline{E},$ where the distance is computed in the euclidean norm.
\end{enumerate}
Then we write $E\in UPC(M,m,d)$. 
\end{dfn}

\begin{rem} Condition (2) clearly implies the second part of condition (1) in the definition. At the same time, the latter implies that in condition (2) one may replace the distance from ${\Rz}^n\setminus E$ by the distance from the boundary $\partial E$. \\ It is also easy to see that a UPC set $E$ is {\it fat}, i.e. $\overline{E}=\overline{\mathrm{int}\ E}$. 
\end{rem}

Since we will be dealing with subanalytic sets let us shortly recall the basic definitions. Semianalytic sets were introduced by S. \L ojasiewicz (see \cite{L}) and led in a natural way to the notion of subanalyticity (the class of subanalytic sets is the smallest boolean algebra closed under projections and including the class of semianalytic sets). For a detailed study of subanalytic geometry we refer the reader to the survey \cite{DS}. 

Let $M$ be a real analytic variety. Throughout this paper we shall deal with $M={\Rz}^m$.

\begin{dfn} (cf. \cite{L}) 
A set $E\subset M$ is called {\it semianalytic}, if for each point $a\in M$ one may choose a neighbourhood $U\ni a$ and a finite family of real analytic functions on $U$, say $f_i, g_{ij}$, $i=1,\dots, r$, $j=1,\dots, s$, such that $$E\cap U=\bigcup_{i=1}^r\{x\in U\mid f_i(x)=0,\> g_{ij}(x)>0,\> j=1,\dots, s\}.$$
\end{dfn}

\begin{dfn}
A set $E\subset M$ is called {\it subanalytic}, if each point $a\in M$ has a neighbourhood $U\ni a$ such that $E\cap U$ is the projection of a relatively compact semianalytic set, i.e. $E\cap U=\pi(A)$, where $\pi\colon M\times N\to M$ is the natural projection, $N$ is a real analytic variety and $A\subset N$ is semianalytic. \\
A mapping $f\colon \om\to N$, where $\om\subset M$ is a set and $M,N$ are real analytic varieties, is called {\it subanalytic}, if its graph is a subanalytic subset of $M\times N$ (then $\om$ also is subanalytic).
\end{dfn}

In \cite{PP} corollary 6.6 states that every relatively compact subanalytic subset $E\subset{\Rz}^n$ which is fat is UPC. As a matter of fact it is a straightforward consequence of theorem 6.4 of \cite{PP} being a refinement of the Bruhat-Cartan-Wallace curve selecting lemma.

\medskip
\noindent{\bf Acknowledgements.} The idea to write this paper came to us logn ago, while listening to the PhD lectures \cite{P} carried out by Professor Wies\l aw Ple\'sniak whom we would like to thank for this inspiration. 

\section{Auxiliary results}

In this part we shall state and prove some facts which in turn will be most helpful in the main part. Let $\pi\colon {\Rz}^m\times {\Rz}^n\ni(x,y)\to x\in{\Rz}^m$ be the natural projection. For a set $A\subset{\Rz}^m\times {\Rz}^n$ we denote its {\sl section} at $x\in\pi(A)$ by $A_x:=\{y\in{\Rz}^n\mid (x,y)\in A\}$. We begin with the following theorem from \cite{LW} (see also \cite{DS}):

\begin{thm}\label{Lojparam} (Regular separation with parameter.)\\
Let $A,B\subset {\Rz}^m\times {\Rz}^n$ two relatively compact subanalytic sets. Then there exists a positive constant $r>0$ such that for all $x\in\pi(A\cap B)$ the inequality 
$$\mathrm{dist}(y, A_x)\geq C(x)\mathrm{dist}(y,A_x\cap B_x)^r,\quad \hbox{\it for}\ y\in B_x\leqno{(\ast)}$$
holds for some $C(x)>0$.\\ Moreover, the function $C\colon \pi(A\cap B)\to (0,+\infty)$ may be required to be subanalytic.
\end{thm}

\begin{rem}
The inequality $(\ast)$ is called {\it \L ojasiewicz inequality} and the sets $A_x$ and $B_x$ satisfying it are said to be {\it regularly separated}.
\end{rem}

From this theorem we derive the following {\L ojasiewicz inequality with parameter}:
\begin{thm}
Let $f\colon \om\to {\Rz}^k$ be a subanalytic function, where $\om\subset{\Rz}^m\times {\Rz}^n$ is relatively compact. Let $f_x(y):=f(x,y)$ for $x\in\pi(\om)$ and $y\in{\om}_x$. Then there exists an exponent $r>0$ such that for each $x\in \pi(\om)$
$$||f_x(y)||\geq C(x)\mathrm{dist}(y,f_x^{-1}(0))^r,\quad\hbox{for}\ y\in {\om}_x$$
for some $C(x)>0$.\\
Moreover the function $C\colon \pi(\om)\to(0,+\infty)$ may be chosen to be subanalytic.
\end{thm}

\begin{proof} Let $A$ be the graph of $f$ and $B:=\om\times \{0\}$. Both are subanalytic subsets of ${\Rz}^m\times {\Rz}^n\times{\Rz}^k$. It is clear that $f^{-1}(0)\times\{0\}=A\cap B$ and $f_x^{-1}(0)\times\{0\}=(A\cap B)_x=A_x\cap B_x$ for $x\in\pi(\om)$. 

Moreover, since $A_x=\{(y,f_x(y))\mid y\in\om_x\}$, one easily sees that $\mathrm{dist}((y,0), A_x)=||f_x(y)||$. 
Finally, $\mathrm{dist}((y,0),A_x\cap B_x)=\mathrm{dist}(y,f_x^{-1}(0))$ and so theorem \ref{Lojparam} gives the result. 
\end{proof}

For the convenience of the reader we will shortly recall the notion of {\it definable cell decomposition} for the o-minimal structure of (totally) subanalytic sets (see e.g. \cite{C}). 

\begin{dfn}
A set $C\subset{\Rz}^m$ is called a {\it (sub)analytic cell} if\\
(i) for $m=1$ the set $C$ is either a point or an open non-empty interval,\\
(ii) for $m>1$ the set $C$ is either the graph
$$C=f=\{(x,t)\in{\Rz}^{m-1}\times{\Rz}\mid x\in C'\ \hbox{\it and}\ t=f(x)\}$$ 
of a continuous subanalytic (and analytic) function $f\colon C'\to {\Rz}$, where $C'$ is a (sub)analytic cell in ${\Rz}^{m-1}$, 
or $C$ is a {\it (sub)analytic prism} 
$$C=(f_1,f_2):=\{(x,t)\in{\Rz}^{m-1}\times{\Rz}\mid x\in C'\ \hbox{\it and}\ f_1(x)<t<f_2(x)\}$$
where $C'$ is a (sub)analytic cell in ${\Rz}^{m-1}$ and $f_j\colon C'\to{\Rz}\cup\{\pm\infty\}$ are two continuous subanalytic (and analytic) functions such that $f_1<f_2$ on $C'$ and $f_j$ has either values in ${\Rz}$ or is constant.
\end{dfn}

Note that an analytic cell is always an analytic manifold. 
Let now $\mathscr{A}$ be a finite family of subanalytic sets in ${\Rz}^m$.

\begin{dfn}
We say that a family of (sub)analytic cells in ${\Rz}^m$, say $\mathscr{C}$, is a {\it (sub)analytic cell decomposition}, if the family $\mathscr{C}$ is finite, its elements are pairwise disjoint, $\bigcup\mathscr{C}={\Rz}^m$ and $\{\pi(C)\mid C\in\mathscr{C}\}$ is a (sub)analytic cell decomposition of ${\Rz}^{m-1}$, where $\pi$ is the natural projection onto the first $m-1$ coordinates. We say that $\mathscr{C}$ is adapted to (or: compatible with) the family $\mathscr{A}$, if for any $C\in\mathscr{C}$, $A\in\mathscr{A}$, one has either $C\subset A$ or $C\cap A=\varnothing$. 
\end{dfn}

Recall that for any finite family of subsets of ${\Rz}^m$, definable in an o-minimal structure on the field ${\Rz}$, there is a definable cell decomposition of ${\Rz}^m$ adapted to that family (for all these notions see e.g. \cite{C}). Then for any fixed $k>0$, taking a refinement of this decomposition we achieve a definable cell decomposition of class $\mathscr{C}^k$. 
Now recall Tamm's lemma (see \cite{DS}):

\begin{lem}
Let $\om\subset{\Rz}^m$ be open and let $f\colon \om\to{\Rz}$ be a subanalytic function, locally bounded in ${\Rz}^m$. Then there exists a $k\in{\Nt}$ such that for each $x\in \om$, if $f$ is of class $\mathscr{C}^k$ in $x$, then $f$ is analytic in $x$.
\end{lem}

From this lemma along with the preceding observation we easily derive the following theorem on analytic cell decomposition. Indeed, for cells given by bounded subanalytic functions one may take a refinement of the decomposition so as to get $\mathscr{C}^k$ functions for an appropriate $k>0$, whence, by Tamm's lemma, analytic cells.

\begin{thm}\label{acd}
Let $E\subset{\Rz}^m$ be a bounded subanalytic set. Then there exists a finite decomposition of $E$ into a disjoint sum of analytic cells. 
\end{thm}

Let us introduce two useful notations for the closed and the open unit cube in ${\Rz}^n$: 
$$
I^n:=\{y\in{\Rz}^n\mid ||y||\leq 1\},\quad J^n:=\{y\in{\Rz}^n\mid ||y||<1\}
$$
and $I:=I^1$, $J:=J^1$. 
Since we will follow the main idea of \cite{PP}, we shall need the following `parameter version' of \cite{PP} proposition 6.3:
\begin{thm}\label{csl}
Let $E\subset{\Rz}_z^k\times{\Rz}_w^n$ be a bounded open subanalytic set and let $\pi(z,w)=z$ be the natural projection. Then 
\begin{enumerate}

\item[(i)] For each point $z\in\pi(E)$ there is a finite number, say $s_z$, of subanalytic mappings $\varphi^z_j\colon I^n\times [-1,1]\to {\Rz}^n$ analytic in the open set $J^n\times(-1,1)$ and satisfying \begin{enumerate}
\item $\varphi_j^z(I^n\times\{t\})\subset E_z$ for $t\in [-1,1]\setminus\{0\}$ and $j=1,\dots,s_z$,

\item $\bigcup_{j=1}^{s_z} \varphi_j^z(I^n\times\{0\})=\overline{E_z}$.
\end{enumerate}

\item[(ii)] There is an integer $s\in{\Nt}$ such that $s_z\leq s$ for all $z\in\pi(E)$. 

\item[(iii)] All the mappings $$\pi(E)\times{I}^n\times{[-1,1]}\ni(z,y,t)\mapsto \varphi_j^z(y,t)\in{\Rz}^n$$ are subanalytic, while $\varphi_j^z(y,\cdot)$ are analytic in $(-1,1)$ for all points $(z,y)\in \pi(E)\times I^n$.
\end{enumerate}
\end{thm}

\begin{rem}\label{obs} Note that adding constant functions we may always assume that $s_z=s$ for all $z\in\pi(E)$.
%
\end{rem}

We shall need the following two lemmata:

\begin{lem}
Let $N\subset{\Rz}_z^k$ be subanalytic and let $F\subset N\times (-c,c)$ be a subanalytic relatively compact set such that $\pi(F)=N$, where $\pi(z,t)=z$ is the natural projection. Then the functions 
$$\kappa\colon N\ni z\mapsto \min F_z\in{\Rz}\quad \hbox{\it and}\quad \lambda\colon N\ni z\mapsto\max F_z\in{\Rz}$$
are subanalytic.
\end{lem}

\begin{proof}
Similarly to \cite{DW} it suffices to observe that 
$$\kappa=\{(z,t)\in N\times(-c,c)\mid \mathrm{dist}(-c,F_z)=c+t\}$$ $$\hbox{and}\>\ \lambda=\{(z,t)\in N\times(-c,c)\mid \mathrm{dist}(c,F_z)=c-t\},$$
since the function $(z,t)\mapsto\mathrm{dist}(t,F_z)$ is subanalytic.
\end{proof}

\begin{lem}\label{lemat}
Let $E\subset{\Rz}_z^k\times{\Rz}_w$ be a bounded analytic cell and $\pi(z,w)=z$ the natural projection. Then there exists a subanalytic function 
$$\zeta\colon \overline{\pi(E)}\times I\times [-1,1]\to {\Rz}$$
which is analytic in $\pi(E)\times J\times (-1,1)$ and such that 
\begin{enumerate}
\item $\zeta(\{z\}\times I\times\{0\})=\overline{E_z}=(\overline{E})_z$ when $z\in\pi(E)$\\ and $\zeta(\{z\}\times I\times\{0\})=(\overline{E})_z$ when $z\in\overline{\pi(E)}\setminus\pi(E)$,
\item $\zeta(\{z\}\times J\times \{t\})\subset E_z$ for all $z\in\pi(E)$, $t\in[-1,1]\setminus\{0\}$,
\item $\zeta(z,u,\cdot)$ is analytic in $(-1,1)$ for all $u\in I$, $z\in\overline{\pi(E)}$. 
\end{enumerate}
\end{lem}

\begin{proof}
We consider two cases:\\
\noindent(i) $E=f$, where $f\colon \pi(E)\to{\Rz}$ is a continuous analytic bounded function over the analytic cell $\pi(E)\subset{\Rz}^k$. Set $\xi'(z,u,t)=f(z)$ for $(z,u,t)\in \pi(E)\times I\times[-1,1]$. Let $N:=\overline{\pi(E)}\setminus\pi(E)$. Note that $\dim N<\dim\pi(E)$ (see \cite{C}). Consider then the intersection of the closure of the graph with $N\times{\Rz}$:
$$F:=(N\times{\Rz})\cap \overline{f}.$$
It is a bounded subanalytic set in ${\Rz}^k\times{\Rz}$ lying all over $N$. Set now 
$$
\kappa\colon N\ni z\mapsto \min F_z\in {\Rz}\quad \hbox{\it and}\quad \lambda\colon N\ni z\mapsto\max F_z\in{\Rz}.
$$
By the previous lemma both these functions are subanalytic and so is 
$$
\xi''(z,u,t)=(1-t^2)\frac{\lambda(z)-\kappa(z)}{2}u+\frac{\lambda(z)+\kappa(z)}{2} 
$$
for $(z,u,t)\in N\times I\times [-1,1]$. The function $\zeta:=\xi'\cup\xi''$ is the sought one.

\noindent(ii) $E=(f_1,f_2)$ is an analytic prism given by two bounded subanalytic and analytic functions over $\pi(E)$ and such that $f_1<f_2$. Then we set 
$$
\xi'(z,u,t)=(1-t^2)\frac{f_2(z)-f_1(z)}{2}u+\frac{f_2(z)+f_1(z)}{2}
$$
for $(z,u,t)\in\pi(E)\times I\times [-1,1]$. Now let 
$$
F:=(N\times{\Rz})\cap \overline{E},
$$
where $N$ is as earlier. Both are relatively compact subanalytic subsets. Let $\kappa,\lambda$ be as earlier. Thus the function
$$
\xi''(z,u,t)=(1-t^2)\frac{\lambda(z)-\kappa(z)}{2}u+\frac{\lambda(z)+\kappa(z)}{2}
$$
for $(z,u,t)\in N\times I\times[-1,1]$ is subanalytic. Then $\zeta:=\xi'\cup\xi''$ is the sought function.
\end{proof}

Now we are ready to prove theorem \ref{csl}:

\begin{proof} {\sl (Theorem \ref{csl}.)} 
We proceed by induction on $n$:\\
(1) First suppose that $n=1$. Then by theorem \ref{acd} 
we may assume that $E$ is an analytic cell. The previous lemma gives the assertion for analytic cells (with $s_z=1$). The theorem for $n=1$ follows now easily.


\smallskip
\noindent (2) Assume $n>1$. As earlier, we may suppose $E\subset{\Rz}_z^k\times{\Rz}^{n-1}_{w'}\times{\Rz}_{w_n}$ is an analytic cell over another analytic cell $E'\subset{\Rz}^k\times{\Rz}^{n-1}$. 
\\
By the induction hypothesis, there are $s$ (cf. remark \ref{obs}) functions $\varphi_j^z\colon {I}^{n-1}\times{[-1,1]}\to {\Rz}^{n-1}$ satisfying our theorem for the set $E'$. 

On the other hand,  whenever $z\in\pi(E)$ is fixed, we have $E_z=p(E\cap (E'_z\times{\Rz}_{w_n}))$, where $p(z,w)=w$ is the natural projection. The set $E_z$ is thus an analytic cell. Thence, if $\zeta(z,\cdot,\cdot,\cdot)$ (with $z\in\pi(E)$) is the function constructed for $E_z$ over $\overline{E'_z}$ as in lemma \ref{lemat}, using the function(s) defining the cell $E$ over $E'$, we have in particular
$$\zeta(\{(z,w')\}\times I\times \{0\})=(\overline{E_z})_{w'}$$
for $z\in\pi(E)$ and $w'\in\overline{E'_z}$. From the construction of $\zeta$ it is easy to see that this function is in fact subanalytic with respect to the variables $(z,w',u,t)$, where $z\in\pi(E)$, $w'\in\overline{E'_z}$ (i.e. in fact $(z,w')\in \overline{E'}\cap (\pi(E)\times {\Rz}_{w'}^{n-1}$), $(u,t)\in I\times[-1,1]$.


If we put now for $(z,y',y_n,t)\in\pi(E)\times I^{n-1}\times I\times [-1,1]$,
$$\psi_j(z,y',y_n,t):=(\varphi_j^z(y',t), \zeta(z,\varphi_j^z(y',t),y_n,t))\in{\Rz}^n,\quad j=1,\dots, s,$$
we obtain a subanalytic function, which is analytic with respect to $(y',y_n,t)\in J^{n-1}\times J\times(-1,1)$, since $\zeta_z$ is analytic in $E'\times J\times (-1,1)$. The theorem follows easily.
%
%
\end{proof}

\section{UPC condition with parameter}

Throughout this section $E\subset{\Rz}_z^k\times{\Rz}_w^n$ is an {\sl open} relatively compact subanalytic set. We denote  $E_z:=\{w\in{\Rz}^n\mid (z,w)\in E\}$ the section of $E$ at $z\in\pi(E)$, where $\pi(z,w)=z$ is the natural projection. 

The aim of this section is to prove the following theorem:

\begin{thm}\label{UPC}
In the introduced setting there exist constants $m,d>0$, $d\in{\Nt}$, and a subanalytic function $M\colon \pi(E)\to(0,+\infty)$ such that for each point $z\in \pi(E)$ $E_z\in UPC(M(z),m,d)$.
\end{thm}

\begin{ex}
In general $M$ cannot be chosen independently of $z\in\pi(E)$. To see this consider the following set:
$$E:=\{(x,y,z)\in (0,1)^3\mid x^2<z^2 y^3\}$$
with $\pi(x,y,z)=z$. Taking for instance $h_0^z(t)=(t/2,zt^2/4)$ (at $0\in\overline{E_z}$) one checks that $E_z\in UPC(z/16, 2,2)$. 
\end{ex}

To prove theorem \ref{UPC} we shall use the ideas of \cite{PP}. Therefore we begin with a counterpart of theorem 6.4 of \cite{PP}, which implies theorem \ref{UPC}:

\begin{thm}
In the introduced setting, there exist constants $d,m>0$, with $d\in{\Nt}$, and a function $M\colon\pi(E)\to(0,+\infty)$ (which may be required to be subanalytic) such that for each $z\in\pi(E)$ there is a mapping $h_z\colon \overline{E_z}\times {\Rz}\to{\Rz}^n$ satisfying the following conditions:\begin{enumerate}

\item[(i)] $h_z(w,t)=\sum_{j=1}^d a^z_j(w)t^j$;

\item[(ii)] $h_z(w,0)=w$, for all $w\in\overline{E_z}$, and $h_z(\overline{E_z}\times(0,1])\subset E_z$;

\item[(iii)] $\mathrm{dist}(h_z(w,t), \partial E_z)\geq M(z)t^m$ for all $(w,t)\in \overline{E_z}\times [0,1]$.
\end{enumerate}
We may require also $h(z,w,t)=h_z(w,t)$ and the coefficients $a^z_j$ be subanalytic.
\end{thm}

\begin{proof} 
For $z\in\pi(E)$ let $\varphi_j^z$, $j=1,\dots, s_z$ be the mappings from theorem \ref{csl}. Adding, if necessary, constant functions, we may assume that $s_z=s$ for all $z\in\pi(E)$. Then we set
$$\psi_j\colon \pi(E)\times I^n\times [0,1]\ni (z,y,t)\mapsto \mathrm{dist}(\varphi_j^z(y,t),\partial E_z)\in\Rz,\quad j=1,\dots, s.$$
All these functions are subanalytic (see e.g. \cite{DS}). Then by theorem \ref{Lojparam} we obtain the inequalities $||(\psi_j)_z(y,t)||\geq C_j(z)\mathrm{dist}((y,t),(\psi_j)_z^{-1}(0))^{r_j}$ with some $r_j, C_j(z)>0$ and for all $(y,t)\in I^n\times [0,1]$, with $z\in \pi(E)$. 

Note that by theorem \ref{csl}(i)(b) $(\psi_j)_z(y,t)=0$ implies $t=0$. Therefore for each $z\in\pi(E)$
$$
||(\psi_j)_z(y,t)||\geq C(z)t^{m},\quad\hbox{\it for}\ (y,t)\in I^n\times[0,1],\leqno{(\ast)}
$$
where now $m:=\max_j r_j$, $C(z):=\min_j C_j(z)$. 

Now fix an integer $d\geq r$ and let $$T_j^z(y,t):=\sum_{\nu=1}^d \frac{1}{\nu!}\frac{\partial^\nu\varphi_j^z}{\partial t^\nu}(y,0)t^\nu,\quad t\in{\Rz}, y\in I^n,$$
be the Taylor polynomial at zero of degree $d$ of the mapping $\varphi_j^z(y,\cdot)$. Since $\varphi_j^z-T_j^z$ is subanalytic in $I^n\times[-1,1]$, analytic in $J^n\times(-1,1)$ and $(d+1)$-flat at zero, then there exists a subanalytic {\sl bounded} mapping $Q_j^z$ such that 
$$\varphi_j^z(y,t)=T_j^z(y,t)+t^{d+1}Q_j^z(y,t),\quad\hbox{\it for}\ (y,t)\in{I}^n\times{[-1,1]}.\leqno{(\ast\ast)}$$
(Moreover, the mappings $Q_j(z,y,t)=Q_j^z(y,t)$ are subanalytic on $\pi(E)\times{I}^n\times{[-1,1]}$.)

Whenever $z\in\pi(E)$ is fixed, one can choose $\delta_z\in(0,1]$ so that for each $j$, $||tQ_j^z(y,t)||\leq C(z)/2$ as $y\in I^n$ and $t\in[0,\delta_z]$. Then by $(\ast)$ and $(\ast\ast)$, for all $j=1,\dots, s$,{\small
$$\mathrm{dist}(T_j^z(y,t),\partial E_z)\geq C(z)t^m-\frac{C(z)}{2}t^d\geq\frac{C(z)}{2} t^m,\quad \hbox{\it for}\ (y,t)\in I^n\times[0,\delta_z].$$}
Thus for $M(z):=C(z)\delta_z^m/2$ we obtain 
$$\mathrm{dist}(T_j^z(y,\delta_zt),\partial E_z)\geq M(z)t^m,\>\ \hbox{\it for}\ (y,t)\in I^n\times[0,1].$$

To finish the proof it remains to observe that $$\bigcup_{j=1}^s T_j^z(I^n\times\{0\})=\bigcup_{j=1}^s\varphi_j^z(I^n\times\{0\})=\overline{E_z}.$$
The construction of $h_z$ consists in taking first a finite disjoint partition of $\overline{E_z}$ by means of the functions $\varphi^z_j$. For if we put $f_j^z(y):=\varphi_j^z(y,0)$ and $F^z_j:=f^z_j(I^n)$, we obtain $\overline{E_z}=\bigcup_j G_j^z$, the union being disjoint, where $G^z_1:=F^z_1$ and $G^z_j=F^z_j\setminus(E^z_1\cup\ldots\cup E^z_j)$ for $j>1$. 

Then each point $w\in\overline{E_z}$ belongs to exactly one set $G^z_j$. We put $h_z(w,t):=\sup T^z_j(y,\delta_z t)$, where the supremum is computed relatively to the lexicographic order over all $y\in (f_j^z)^{-1}(w)$ (which forms a compact subanalytic subset of $I^n$). It is easy to see that $\pi(E)\ni z\mapsto\delta_z$ may always be required to be subanalytic too, which ends the proof.
\end{proof}

Clearly, theorem \ref{UPC} is a direct consequence of the one above. 

\smallskip
Following the proof of \cite{PP} theorem 3.1 one can now esily obtain the following Markov's inequality with parameter for subanalytic sets. Let $||P||_K:=\sup\{|P(w)|\mid w\in K\}$ for any bounded $K\subset{\Rz}^n$ and a continuous function $P$.

\begin{thm}
In the introduced setting there exists a constant $r>0$ such that for each polynomial $P\colon{\Rz}^n\to {\Rz}$ of degree $\leq q$ and for each $\alpha\in\mathbb{Z}^n_+$, one has 
$$||D^\alpha P||_{E_z}\leq C(z,E,\alpha)q^{r|\alpha|} ||P||_{E_z},\quad z\in\pi(E),$$
where $C(z,E,\alpha)>0$ is a constant depending only on the set $E$, the point $z\in\pi(E)$ and the multi-index $\alpha$. 
\end{thm}

\section{Acknowledgements}
The second-named author was partially supported by  Polish Ministry of Science
and Higher Education grant 1095/MOB/2013/0.

\bigskip
\noindent{\small Jagiellonian University\hfill Cracow University of Economics\\
Institute of Mathematics\hfill Department of Mathematics\\
\L ojasiewicza 6\hfill Rakowiecka 27\\
30-348 Krak\'ow, Poland\hfill 31-510 Krak\'ow, Poland\\
e-mail addresses: {\tt denkowsk@im.uj.edu.pl}\hfill{\tt anna.denkowska@uek.krakow.pl}}


\begin{thebibliography}{\L W}


\bibitem[C]{C} M. Coste, {\it An introduction to o-minimal geometry}, pdf file available on the RAAG web page http://www.ihp-raag.org;

\bibitem[DS]{DS} Z. Denkowska, J. Stasica, {\it Ensembles sous-analytiques \`a la polonaise}, preprint (1985)  Universit\'e d'Angers 2005;

\bibitem[DW]{DW} Z. Denkowska, K. Wachta, {\it On the existence of a subanalytic selection for the subanalytic set-valued function}, UIAM fasc. XXVII (1988), pp. 187-188;

\bibitem[\L]{L} S. \L ojasiewicz, {\it Ensembles semi-analytiques}, preprint IHES 1965;

\bibitem[\L W]{LW} S. \L ojasiewicz, K. Wachta, {\it S\'eparation r\'eguli\`ere avec param\`etre pour les sous-analytiques}, Bull. Pol. Acad. Sci. XXX nr 7-8, (1982), pp. 325-328; 

\bibitem[PP]{PP} W. Paw\l ucki, W. Ple\'sniak, {\it Markov's inequality and $\mathscr{C}^\infty$ functions on sets with polynomial cusps}, Math. Ann. 275, (1986), pp. 467-480;

\bibitem[P]{P} W. Ple\'sniak, {\it Polynomial inequalities}, lecture for PhD students at the Jagellonian University, Cracow 2005.

\end{thebibliography}
\end{document}